\documentclass[reqno, 12pt,a4letter]{amsart}
\usepackage{amsmath,amsxtra,amssymb,latexsym, amscd,amsthm}
\usepackage{graphicx,color}
\usepackage[active]{srcltx}
\usepackage[utf8]{inputenc}
\usepackage[mathscr]{euscript}
\usepackage{mathrsfs}
\usepackage{enumerate}
\usepackage[english]{babel}
\usepackage{color}
\newtheorem {theorem}{Theorem}[section]

\newtheorem {lemma}{Lemma}[section]
\newtheorem {example}{Example}[section]
\newtheorem {definition}{Definition}[section]
\newtheorem {remark}{Remark}[section]

\def\ees{{\accent"5E e}\kern-.385em\raise.2ex\hbox{\char'23}\kern-.08em}
\def\EES{{\accent"5E E}\kern-.5em\raise.8ex\hbox{\char'23 }}
\def\ow{o\kern-.42em\raise.82ex\hbox{
\vrule width .12em height .0ex depth .075ex \kern-0.16em \char'56}\kern-.07em}
\def\OW{O\kern-.460em\raise1.36ex\hbox{
\vrule width .13em height .0ex depth .075ex \kern-0.16em \char'56}\kern-.07em}

\title[]{Error bounds of regularized gap functions for polynomial variational inequalities}

\author{DINH BUI VAN$^{\dagger}$}
\address{Department of Mathematics, Le Quy Don Technical University, No 236 Hoang Quoc Viet road, Hanoi, Vietnam}
\email{vandinhb@gmail.com}

\author{TI\EES N-S\OW N PH\d{A}M$^{\ddagger}$}
\address{Department of Mathematics, University of Dalat, 1 Phu Dong Thien Vuong, Dalat, Vietnam}
\email{sonpt@dlu.edu.vn}

\thanks{$^{\ddag, \ddagger}$The authors are partially supported by Vietnam National Foundation for Science and Technology Development (NAFOSTED), grant 101.04-2019.302.}

\keywords{Variational inequality, Regularized gap function, Error bound, {\L}ojasiewicz inequality, Polynomial}

\subjclass{90C26, 65H10, 90C26, 90C31, 26C05}

\date{ \today}

\begin{document}
\maketitle

\begin{abstract}
This paper is devoted to present new error bounds of regularized gap functions for polynomial variational inequalities
with exponents explicitly determined by the dimension of the underlying space and the number/degree of the involved polynomials. The developed techniques are largely based on variational analysis and semialgebraic geometry, which allow us to establish a nonsmooth extension of the seminal \L ojasiewicz's gradient inequality to regularized gap functions with explicitly calculated exponents. 
\end{abstract}

\section{Introduction}

We study the {\em variational inequality} in which a point $x \in \Omega$ is sought such that
\begin{equation} \label{VI}
\langle F(x), y - x \rangle \ge 0 \quad \textrm{ for all } \quad y \in \Omega, \tag{VI}
\end{equation}
where $F \colon \mathbb{R}^n \rightarrow \mathbb{R}^n$ is a map and $\Omega$ is a nonempty closed convex subset of $\mathbb{R}^n.$ When $\Omega$ is the nonnegative orthant in $\mathbb{R}^n,$ the problem \eqref{VI} reduces to the nonlinear complementary problem.

Variational inequalities have been widely studied in various fields such as mathematical programming, game theory and economics, etc.
There is a large literature on all aspects of the theory and application of variational inequalities; for more details, we refer the reader to the survey by Harker and Pang \cite{Harker1990} and the comprehensive monograph by Facchinei and Pang \cite{Facchinei2003} with the references therein.

Many fruitful approaches to both theoretical and numerical treatment of variational inequalities make use of merit functions. 
One such function is the {\em regularized gap function} $\psi \colon \mathbb{R}^n \to \mathbb{R}$ defined by
\begin{eqnarray*}
\psi(x) &:=& \sup_{y \in \Omega} \big \{\langle F(x), x - y \rangle - \frac{\rho}{2} \|x - y\|^2 \big\},
\end{eqnarray*}
where $\rho$ is a positive real number (cf. \cite{Auchmuty1989, Facchinei2003, Fukushima1992, Solodov2003, Taji1993, Wu1993}).  This function has a number of interesting properties. For example, it is finite valued everywhere, nonnegative on $\Omega,$ and becomes zero at any solution of the problem~\eqref{VI}. Furthermore, it is continuously differentiable whenever $F$ is continuously differentiable. These properties are basic for development of iterative decent algorithms for solving variational inequalities.

On the other hand, the theory of error bounds provides a useful aid for understanding the connection between a merit function and the actual distance to its zero set, and hence plays  an  important  role  in  convergence  analysis  and  stopping criteria  for  many iterative  algorithms; for more details, see \cite[Chapter~6]{Facchinei2003} and references therein. Therefore, it would be interesting and useful to investigate error bounds for regularized gap functions associated with variational inequalities.

Assume that the map $F$ is strongly monotone. By virtue of the consideration of differentials, Wu et al. \cite{Wu1993}, Yamashita et al. \cite{Yamashita1997-2}, and Huang et al. \cite{Huang2005} have addressed the error bound issues for $\psi$ when $F$ is smooth, and thereby established convergence results of sequences obtained by an algorithm of Armijo type. These results are extended by Ng and Tan \cite{Ng2007} to cover the case that $F$ is not necessarily smooth; see also \cite{LiGuoyin2009, LiGuoyin2010-1, Solodov2003, Yamashita1997-1} for related works.

We now assume that $F$ is a polynomial map and $\Omega$ is a closed set defined by finitely many polynomial equalities and inequalities. Thanks to the classical \L ojasiewicz inequality (see Theorem~\ref{ClassicalLojasiewiczInequality} in the next section), we can see that for any compact set $K \subset \Bbb{R}^n,$ there exist constants $c > 0$ and $\alpha > 0$ satisfying the following error bound
\begin{eqnarray*}
c\, \mathrm{dist}(x, [x \in \Omega : \psi(x) = 0]) & \leq & [\psi (x)]^{\alpha} \quad \textrm{ for all }  \quad x \in \Omega \cap K,
\end{eqnarray*}
where $\mathrm{dist}(\cdot, \cdot)$ stands for the usual Euclidean distance function. 

In the spirit of \cite{Acunto2005, Kurdyka2014, LiGuoyin2010-2, LiGuoyin2018, LiGuoyin2015, PHAMTS2018}, we show in this paper that the exponent $\alpha$ in the above error bound result is explicitly determined by the dimension of the underlying space and the number/degree of the involved polynomials. The main techniques used in this paper are largely based on variational analysis and semialgebraic geometry, which allow us to establish a nonsmooth extension of the seminal \L ojasiewicz's gradient inequality to the regularized gap function with explicitly calculated exponent. 
It is worth emphasizing that error bound results with {\em explicit} exponents are indeed important for both theory and applications since they can be used, e.g., to establish {\em explicit convergence rates} of iterative algorithms for the solution of variational inequalities.

Note that we do not assume that the map $F$ is (strongly) monotone or the constraint set $\Omega$ is bounded. Furthermore, while all results are stated for the regularized gap function, we believe analogous results can be obtained for the so-called {\em D-gap function}; for the definition and properties of this function, we refer to \cite{Facchinei2003}. However, to lighten the exposition, we do not pursue this idea here.

The rest of the paper is organized as follows: In Section~\ref{SectionPreliminaries}, we review some preliminaries from variational analysis and semi-algebraic geometry that will be used later. In Section~\ref{Section3}, we provide a nonsmooth version of \L ojasiewicz's gradient inequality to the regularized gap function. Finally, in Section~\ref{Section4}, we establish major error bounds 
of regularized gap functions for polynomial variational inequalities.

\section{Preliminaries} \label{SectionPreliminaries}

Throughout this work we deal with the Euclidean space $\mathbb{R}^n$ equipped with the usual scalar product $\langle \cdot, \cdot \rangle$ and the corresponding Euclidean norm $\| \cdot\|.$ The distance from a point $x \in \mathbb{R}^n$ to a nonempty set $A \subset \mathbb{R}^n$ is defined by
$$\mathrm{dist}(x, A) := \inf_{y\in A} \|x - y\|.$$
By our convention, the distance to the empty set is defined as $\mathrm{dist}(x, \emptyset) = 1.$ We write ${\mathrm{co}}A$ for the convex hull of $A.$ We denote by $\mathbb{B}_r(x)$ the closed ball centered at $x$ with radius $r;$ we also use the notations  $\mathbb{B}_r$ for $\mathbb{B}_r(0)$ and $\mathbb{B}$ for the closed unit ball. For each real number $r,$ we put $[r]_+ := \max\{r, 0\}.$

\subsection{Some subdifferentials}
We first recall the notions of subdifferentials, which are crucial for our considerations. For nonsmooth analysis we refer the reader to the comprehensive texts \cite{Clarke1983, Mordukhovich2006-1, Rockafellar1998}.

\begin{definition}{\rm 
Let $f \colon {\Bbb R}^n \rightarrow {\Bbb R}$ be a lower semicontinuous function and $x \in {\Bbb R}^n$.
\begin{enumerate}
  \item[(i)] The {\em Fr\'echet subdifferential} $\hat{\partial} f(x)$ of $f$ at $x $ is given by
$$\hat{\partial} f(x) := \left \{ v \in {\Bbb R}^n \ : \ \liminf_{\| h \| \to 0, \ h \ne 0} \frac{f(x + h) - f(x) - \langle v, h \rangle}{\| h \|} \ge 0 \right \}.$$

  \item[(ii)] The {\em limiting} (known also as  {\it basic, Mordukhovich}) {\it subdifferential} of $f$  at $x,$ denoted by ${\partial} f(x),$ is the set of all cluster points of sequences $\{v^l\}$ such that $v^l\in \hat{\partial} f(x^l)$ and $(x^l, f(x^l)) \to (x, f(x))$ as $l \to \infty.$

\item[(iii)] Assume that $f$ is locally Lipschitz. By Rademacher's theorem, $f$ has at almost all points $x \in \mathbb{R}^n$ a gradient, which we denote $\nabla f(x).$ Then the {\em Clarke (or convexified) subdifferential} ${\partial}^\circ f(x)$ of $f$ at $x$ is defined by
$${\partial}^\circ f(x) := {\mathrm{co}}\{\lim \nabla f(x^k) \ : \ x^k\to x\}.$$
\end{enumerate}
}\end{definition}

\begin{remark}{\rm 
It is  well-known from variational analysis (see e.g., \cite{Clarke1983, Rockafellar1998}) that
\begin{enumerate}
  \item[(i)] $\hat{\partial} f(x)$ (and a fortiori $\partial f(x)$) is nonempty in a dense subset of the domain of $f.$
  \item[(ii)] If $f$ is locally Lipschitz, then ${\partial}^\circ (-f)(x)=-{\partial}^\circ f(x)$ and
$$\hat{\partial}f(x)\subset {\partial}f(x)\subset {\partial}^\circ f(x) = {\mathrm{co}} {\partial} f(x).$$
\end{enumerate}
}\end{remark}

\subsection{Semialgebraic geometry}
In this subsection, we recall some notions and results of semialgebraic geometry, which can be found in \cite{Bochnak1998} or \cite[Chapter 1]{HP2017}.

\begin{definition}{\rm
\begin{enumerate}
  \item[(i)] A subset of $\mathbb{R}^n$ is called {\em semialgebraic} if it is a finite union of sets of the form
$$\{x \in \mathbb{R}^n \ : \ f_i(x) = 0, i = 1, \ldots, k; f_i(x) > 0, i = k + 1, \ldots, p\}$$
where all $f_{i}$ are polynomials.
 \item[(ii)]
Let $A \subset \Bbb{R}^n$ and $B \subset \Bbb{R}^p$ be semialgebraic sets. A map $f \colon A \to B$ is said to be {\em semialgebraic} if its graph
$$\{(x, y) \in A \times B \ : \ y = f(x)\}$$
is a semialgebraic subset of $\Bbb{R}^n\times\Bbb{R}^p.$
\end{enumerate}
}\end{definition}

The class of semialgebraic sets is closed under taking finite intersections, finite unions, and complements; furthermore, a Cartesian product of semialgebraic sets is semialgebraic. A major fact concerning the class of semialgebraic sets is given by the following seminal result of semialgebraic geometry.

\begin{theorem}[Tarski--Seidenberg theorem] \label{TarskiSeidenbergTheorem}
Images of semialgebraic sets under semialgebraic maps are semialgebraic.
\end{theorem}

\begin{remark}\label{TarskiSeidenbergRemark} {\rm
As an immediate consequence of Tarski--Seidenberg Theorem, we get semialgebraicity of any set $\{ x \in A : \exists y \in B,  (x, y) \in C \},$  provided that $A ,  B,$  and $C$  are semialgebraic sets in the corresponding spaces.  It follows that also $\{ x \in A : \forall y \in B,  (x, y) \in C \}$ is
a semialgebraic set as its complement is the union of the complement of $A$  and the set
$\{ x \in A : \exists y \in B,  (x, y) \not\in C \}.$ Thus, if we have a finite collection of semialgebraic sets,
then any set obtained from them with the help of a finite chain of quantifiers is also semialgebraic. 
}\end{remark}

\begin{theorem}[the classical \L ojasiewicz inequality] \label{ClassicalLojasiewiczInequality}
Let $f \colon {\Bbb R}^n \rightarrow {\Bbb R}$ be a continuous semialgebraic function. For each compact subset $K$ of $\Bbb{R}^n$ with $K \cap f^{-1}(0) \ne \emptyset,$ there exist constants $c>0$ and $\alpha > 0$ satisfying the following inequality
\begin{eqnarray*}
c\, \mathrm{dist}(x, f^{-1}(0)) & \le & |f(x)|^\alpha  \quad \textrm{ for all } \quad x \in K.
\end{eqnarray*}
\end{theorem}

We also need another fundamental result taken from \cite[Theorem~4.2]{Acunto2005}, which provides an exponent estimate in the \L ojasiewicz gradient inequality for polynomials.

\begin{theorem}[\L ojasiewicz's gradient inequality] \label{LojasiewiczGradientInequality}
Let $f \colon \mathbb{R}^n \rightarrow \mathbb{R}$ be a polynomial of degree $d \ge 1$ and let $\bar{x} \in \mathbb{R}^n.$ Then there exist positive constants $c$ and $\epsilon$ such that
$$\|\nabla f(x)\|\ge c|f(x) - f(\bar x)|^{1 - \frac{1}{\mathscr{R}(n, d)}} \quad \textrm{ for all } \quad x \in \mathbb{B}_\epsilon(\bar x).$$
Here and in the following, we put
\begin{eqnarray}\label{RFunction}
\mathscr{R}(n, d) :=
\begin{cases}
d(3d - 3)^{n - 1} &\ {\rm if} \ d \geq 2, \\
1 &\  {\rm if} \ d = 1.
\end{cases}
\end{eqnarray}
\end{theorem}

\section{The \L ojasiewicz gradient inequality for the regularized gap function} \label{Section3}

In this section, we establish a nonsmooth version of the \L ojasiewicz's gradient inequality with explicitly calculated exponents to regularized gap functions of polynomial variational inequalities.

From now on, let $F \colon \mathbb{R}^n \rightarrow \mathbb{R}^n$ be a polynomial map of degree at most $d \ge 1.$ Let $g_i, h_j \colon \mathbb{R}^n \to \mathbb{R}$ 
for $i=1, \ldots, r$ and $j=1, \ldots, s$ be polynomial functions of degree at most $d,$ and assume that the set
\begin{eqnarray*}
\Omega &:=& \{x \in \mathbb{R}^n \ : \ g_i(x) \leq 0, \ i=1, \ldots, r, \ h_j(x) = 0, \ j=1, \ldots, s\}
\end{eqnarray*}
is (not necessarily convex or bounded) nonempty. Recall the variational inequality formulated in the introduction section: find a point $x \in \Omega$ such that
\begin{equation} \label{VI}
\langle F(x), y - x \rangle \ge 0 \quad \textrm{ for all } \quad y \in \Omega. \tag{VI}
\end{equation}

Fix a positive real number $\rho$ and define the function $\phi \colon \mathbb{R}^n \times \mathbb{R}^n \rightarrow \mathbb{R}, (x, y) \mapsto \phi(x, y),$  by
\begin{eqnarray*}
\phi(x, y) &:=& \langle F(x), x - y \rangle - \frac{\rho}{2} \|x - y\|^2.
\end{eqnarray*}
By definition, $\phi$ is a polynomial in $2n$ variables of degree at most $d + 1.$ Furthermore, for each $x \in \mathbb{R}^n$ we have 
$\lim_{\|y\| \to \infty} \phi(x, y) = - \infty,$ and so the regularized gap function associated to the problem~\eqref{VI}:
$$\psi \colon \mathbb{R}^n \rightarrow \mathbb{R}, \quad x \mapsto \sup_{y \in \Omega} \phi(x, y),$$ 
is well-defined. We will write
\begin{eqnarray*}
\Omega(x) := \{y \in \Omega \ : \  \psi(x) = \phi(x, y)\} \quad \textrm{ for } \quad x \in \mathbb{R}^n.
\end{eqnarray*}

\begin{lemma} \label{Lemma31}
The following statements hold
\begin{enumerate}
\item [(i)] For each $x \in \mathbb{R}^n,$ $\Omega(x)$ is a nonempty compact set.
\item [(ii)] Let $\bar{x} \in \mathbb{R}^n.$ For any $\epsilon > 0,$ there exists a constant $R > 0$ such that
\begin{equation*}
\Omega(x) \subset \{y \in \mathbb{R}^n  :  \| y \| < R \} \quad \textrm{ for all } \quad x \in \mathbb{B}_\epsilon(\bar{x}).
\end{equation*}
\end{enumerate}
 \end{lemma}
\begin{proof}
The lemma follows immediately from the fact that for each $x \in \mathbb{R}^n,$ 
$$\lim_{\|y\| \to \infty} \phi(x, y) = - \infty.$$
The details are left to the reader.
\end{proof}

By Lemma~\ref{Lemma31}, we can write $\psi(x) = \max_{y \in \Omega} \phi(x, y).$ Furthermore, thanks to Theorem~\ref{TarskiSeidenbergTheorem} (see also Remark~\ref{TarskiSeidenbergRemark}), it is not hard to check that the function $\psi$ is semialgebraic.

\begin{lemma} \label{Lemma32}
The function $\psi$ is locally Lipschitz and satisfies
\begin{eqnarray*}
\partial ^0 \psi (x)  &  = & \mathrm{co} \left \{ \nabla_x \phi(x, y) \ : \ y \in \Omega(x) \right\} \quad \textrm{for all } \quad x \in \mathbb{R}^n,
\end{eqnarray*}
where $\nabla_x \phi$ is the derivative of $\phi$ with respect to $x.$ In particular, $\partial^\circ \psi (x)$ is a nonempty, compact and convex set.
\end{lemma}
\begin{proof}
Let $\bar{x} \in \mathbb{R}^n$ and $\epsilon > 0.$ By Lemma~\ref{Lemma31}, there exists $R > 0$ such that $\|y\| < R$ for all $y \in \Omega(x)$ and all $x \in \mathbb{B}_\epsilon(\bar{x}).$ Hence we can write 
$$\psi(x) = \max_{y \in \Omega, \|y\| \le R} \phi(x, y) \quad \textrm{ for all } \quad x \in \mathbb{B}_\epsilon(\bar{x}).$$
This implies easily that $\psi$ is locally Lipschitz. Finally, the formula for $\partial ^0 \psi$ follows immediately from \cite[Theorem~2.1]{Clarke1975}.
\end{proof}

We need the following qualification condition imposed on the constraint set $\Omega.$

\begin{definition}{\rm 
We say that the {\em Mangasarian--Fromovitz constraint qualification} (MFCQ) holds on $\Omega$ if, for every $x \in \Omega,$ the gradient vectors $\nabla h_j(x), j = 1, \ldots, s,$ are linearly independent and there exists a vector $v \in \Bbb{R}^n$ such that $\langle \nabla g_i(x), v \rangle < 0, i \in \{i : g_i(x) = 0\}$ and $\langle \nabla h_j(x), v \rangle  = 0, j = 1, \ldots, s.$
}\end{definition}

For each $x \in \Omega,$ we let 
\begin{eqnarray*}
N(\Omega, x) &:=& \left \{\sum_{i=1}^r \mu_i \nabla g_i(x) +  \sum_{j=1}^s\kappa_j\nabla h_j(x) \ : \ \mu_i, \kappa_j \in \mathbb{R}, \mu_i \ge 0, \ \mu_i g_i(x) = 0, \ i = 1, \ldots, r \right \}.
\end{eqnarray*}
One can check that the set $N(\Omega, x)$ is a convex cone and if (MFCQ) holds, then $N(\Omega, x)$  is a closed set.

We are ready to formulate a nonsmooth version of \L ojasiewicz's gradient inequality with explicit exponent for the regularized gap function $\psi,$
which plays a key role in establishing our error bounds (see Theorems~\ref{Theorem41}~and~\ref{Theorem42} below).

\begin{theorem} \label{Theorem31}
Assume that \emph{(MFCQ)} holds on $\Omega.$ For each $\bar{x} \in \Omega,$ there exist constants $c > 0$ and $\epsilon > 0$ such that for all $x \in \Omega \cap \mathbb{B}_{\epsilon}(\bar{x}),$
\begin{eqnarray} \label{PT2}
\inf\{\|w\| \ : \ w \in \partial^0 \psi(x) + N(\Omega, x) \} &\ge& c |\psi(x) - \psi(\bar{x})|^{1 - \alpha},
\end{eqnarray}
where $\alpha := \frac{1}{\mathscr{R}(n(n + 3) + r(n + 2) + s(n + 2), d + 2)}$ and the function $\mathscr{R}(\cdot, \cdot)$ is defined in \eqref{RFunction}.
\end{theorem}

The proof of Theorem~\ref{Theorem31} will be divided into several steps, which are summarized as follows:
\begin{enumerate}[{\rm (1)}]
\item Prove the set-valued map $\mathbb{R}^n \rightrightarrows \mathbb{R}^n, x \mapsto \Omega(x),$ and certain Lagrange multipliers are upper H\"older continuous.
\item Estimate from above the Clarke subdifferential $\partial^\circ \psi(x).$
\item Construct explicitly a polynomial function $P$ based on this estimate and the definition of the cone $N(\Omega, x).$
\item Prove the inequality~\eqref{PT2} by applying Theorem~\ref{LojasiewiczGradientInequality} to $P.$
\end{enumerate}

We first show the upper H\"older continuity of the set-valued map $\mathbb{R}^n \rightrightarrows \mathbb{R}^n, x \mapsto \Omega(x).$

\begin{lemma} \label{Lemma33}
Let $\bar{x} \in \mathbb{R}^n.$ For each $\epsilon > 0$ there exist constants $c > 0$ and $\alpha > 0$ such that
\begin{eqnarray*}
\Omega(x) \subset \Omega(\bar{x}) + c\|x - \bar{x}\| ^\alpha \mathbb{B} \quad \textrm{ for all } \quad x \in \mathbb{B}_\epsilon(\bar{x}).
\end{eqnarray*}
\end{lemma}

\begin{proof}
Define the function $\Gamma \colon \mathbb{R}^n \times \mathbb{R}^n \rightarrow \mathbb{R}$ by
$$\Gamma(x, y) := [\psi(x) - \phi(x, y)]_+ + \sum_{i = 1}^r [g_i(y)]_+ + \sum_{j = 1}^s |h_j(y)|.$$
It is easy to check that $\Gamma$ is locally Lipschitz and semialgebraic. Furthermore, we have
\begin{eqnarray*}
\Omega(x)
& = &  \{y \in \Omega \ : \ \psi(x) - \phi(x, y) =  0\} \\
& = &  \{y \in \mathbb{R}^n \ : \ \Gamma(x, y) = 0\}.
\end{eqnarray*}

Let $\epsilon > 0.$ By Lemma \ref{Lemma31}, there exists a constant $R > 0$ such that $\Omega(x) \subset \mathbb{B}_R$ for all $x \in \mathbb{B}_\epsilon(\bar{x}).$ Since $\mathbb{B}_R$ is a compact set, it follows from the classical \L ojasiewicz inequality (see Theorem~\ref{ClassicalLojasiewiczInequality}) that there are constants $c > 0$ and $\alpha > 0$ such that
\begin{eqnarray*}
c\, \mathrm{dist}(y , \Omega(\bar{x})) & \le & |\Gamma(\bar{x}, y)|^{\alpha} \quad \textrm{ for all } \quad y \in \mathbb{B}_R.
 \end{eqnarray*}

On the other hand, since $\Gamma$ is locally Lipschitz, it is globally Lipschitz on the compact set $\mathbb{B}_\epsilon(\bar{x}) \times \mathbb{B}_R;$  in particular, there exists a constant $L > 0$ such that
 \begin{eqnarray*}
|\Gamma(x, y) - \Gamma(\bar{x}, y)|  & \le & L \|x - \bar{x}\| \quad \textrm{ for all } \quad (x, y) \in \mathbb{B}_\epsilon(\bar{x}) \times \mathbb{B}_R.
\end{eqnarray*}

Let $x \in \mathbb{B}_\epsilon(\bar{x}),$ and take an arbitrary  $y \in \Omega(x).$ Then $y \in \mathbb{B}_R$ and $\Gamma(x, y) = 0.$ Therefore,
 \begin{eqnarray*}
c\, \mathrm{dist}(y , \Omega(\bar{x}))
& \le & |\Gamma(\bar{x}, y)|^{\alpha} \ = \ |\Gamma(x, y) - \Gamma(\bar{x}, y)|^{\alpha} \\
& \le & L^{\alpha} \|x - \bar{x}\|^{\alpha}.
 \end{eqnarray*}
This implies immediately the required statement.
\end{proof}

The next three lemmas provide estimates for the Fr\'echet, limiting and Clarke subdifferentials of the function $\psi.$

\begin{lemma} \label{Lemma34}
Assume that \emph{(MFCQ)} holds on $\Omega.$ Let ${x} \in \mathbb{R}^n.$  For each $y \in \Omega(x),$ it holds that
\begin{eqnarray*}
\hat \partial (-\psi) (x)
& \subset & \left \{ v \ : \ (v, 0) \in - \nabla \phi(x, y) + \{0\} \times N(\Omega, y) \right\}.
\end{eqnarray*}
\end{lemma}
\begin{proof}
Let $y \in \Omega(x).$ Take arbitrary $v \in \hat\partial (-\psi) (x)$ and $\epsilon >0.$ By the definition of the Fr\'echet subdifferential, there exists a constant $\delta >0$ such that
\begin{eqnarray*}
-\psi (x') + \psi (x) - \langle v, x' - x\rangle &\geq & - \epsilon \|x' - x\| \quad \textrm{ for all } \quad x' \in \mathbb{B}_{\delta}(x).
\end{eqnarray*}
Then for any $(x', y')\in \mathbb{B}_{\delta}(x)\times \Omega$, we have
\begin{eqnarray*}
-\phi (x',y')-\langle v, x' - x\rangle + \epsilon \|x'-x\| &\geq& -\psi (x')-\langle v, x'-x\rangle +\epsilon \|x'-x\| \\
&\geq& -\psi (x) \ = \ -\phi (x,y),
\end{eqnarray*}
which yields that $(x,y)$ is a minimizer of the (locally Lipschitz) function 
$$\mathbb{B}_{\delta}(x)\times \Omega \rightarrow \mathbb{R}, \quad (x',y') \mapsto -\phi (x',y')-\langle v, x'-x\rangle +\epsilon \|x'-x\|.$$ 
By Lagrange's multipliers theorem and  the sum rule (see, for example, \cite[Theorem 3.36]{Mordukhovich2006-1}), we have
\begin{eqnarray*}
(0, 0) & \in & -\nabla \phi (x, y) - (v, 0)+\epsilon (\mathbb{B} \times \{0\}) + \{0\}\times N(\Omega,y).
\end{eqnarray*}
Letting $\epsilon \to 0$ yields
\begin{eqnarray*}
(v, 0) & \in & -\nabla \phi (x, y)  + \{0\}\times N(\Omega, y),
\end{eqnarray*}
which completes the proof.
\end{proof}

\begin{lemma} \label{Lemma35}
Assume that \emph{(MFCQ)} holds on $\Omega.$ For all $x \in \mathbb{R}^n,$ we have
\begin{eqnarray*}
\partial (-\psi) (x)
& \subset & \cup_{y \in \Omega(x)}\left \{ v : (v, 0) \in - \nabla \phi(x, y) + \{0\} \times N(\Omega, y) \right\}.
\end{eqnarray*}
\end{lemma}

\begin{proof}
Let $v \in \partial (-\psi)(x).$ The definition of the limiting subdifferential gives us the existence of sequences $\{x^k\}_{k \in \mathbb{N}}  \subset  \mathbb{R}^n$ and $\{v^k\}_{k \in \mathbb{N}} \subset \hat{\partial} (-\psi)(x^k)$ with 
\begin{eqnarray*}
\lim_{k \to \infty} x^k &=& x \quad  \textrm{ and } \quad \lim_{k \to \infty}v^k \ = \ v.
\end{eqnarray*}
For each integer number $k,$ take any $y^k \in \Omega(x^k).$ By Lemmas~\ref{Lemma31}~and~\ref{Lemma33} (and by choosing a subsequence, if necessary) we may assume, without loss of generality, that there exists $y \in \Omega(x)$ such that $y = \lim_{k \to \infty} y^k.$ On the other hand, by Lemma~\ref{Lemma34}, we have
$$(v^k, 0) \in -\nabla \phi(x^k, y^k) + \{0\} \times N(\Omega, y^k).$$
Letting $k$ tend to infinity, we get
$$(v, 0) \in -\nabla \phi(x, y) + \{0\} \times N(\Omega, y),$$ 
and so the desired conclusion follows.
\end{proof}

\begin{lemma} \label{Lemma36}
Assume that \emph{(MFCQ)} holds on $\Omega.$ For all $x \in \mathbb{R}^n,$ we have
\begin{eqnarray*}
\partial^\circ \psi (x) & \subset  & \mathrm{co}\left(\cup_{y \in \Omega(x)}\left \{ v : (v, 0) \in \nabla \phi(x, y) - \{0\} \times N(\Omega, y) \right\} \right).
\end{eqnarray*}
\end{lemma}
\begin{proof}
Indeed, it follows from the definitions that
$$\partial^\circ \psi (x) \ = \ - \partial^\circ (-\psi) (x) \ = \ - \mathrm{co}(\partial(-\psi) (x)).$$
This combined with Lemma~\ref{Lemma35} leads to the desired assertion.
\end{proof}

The following lemma is simple but useful.
\begin{lemma}\label{Lemma37}
Assume that \emph{(MFCQ)} holds on $\Omega.$ Let $\{x^k\}_{k \in \mathbb{N}} \subset \Omega$ and $\{v^k\}_{k \in \mathbb{N}} \subset N(\Omega, x^k)$ be two bounded sequences such that
\begin{eqnarray*}
v^k &=& \sum_{i=1}^r \mu_i^k \nabla g_i(x^k) +  \sum_{j=1}^s \kappa_j^k \nabla h_j(x^k), \\
0 &=& \mu_i^k g_i(x^k), \ \mu_i^k \ge 0,  \ \textrm{ for } \ i = 1, \ldots, r,
\end{eqnarray*}
for some $\mu^k := (\mu_1^k, \ldots, \mu_r^k) \in \mathbb{R}^r$ and $\kappa^k := (\kappa_1^k, \ldots, \kappa_s^k) \in \mathbb{R}^s.$ Then the sequences
$\{\mu^k\}_{k \in \mathbb{N}}$ and $\{\kappa^k\}_{k \in \mathbb{N}}$ are bounded.
\end{lemma}
\begin{proof}
Arguing by contradiction, assume that
\begin{eqnarray*}
\lim_{k \to \infty} \|(\mu^k, \kappa^k)\| & = & +\infty.
\end{eqnarray*}
Passing to a subsequence if necessary we may  assume that the following limits exist:
\begin{eqnarray*}
(\bar{\mu}, \bar{\kappa}) &:=&  \lim_{k \to \infty} \frac{(\mu^k, \kappa^k)}{\|(\mu^k, \kappa^k)\|} \in \mathbb{R}^r \times \mathbb{R}^s,\\
\bar{x} &:=&  \lim_{k \to \infty} x^k \in \Omega.
\end{eqnarray*}
Then we have $(\bar{\mu}, \bar{\kappa}) \ne (0, 0)$ and
\begin{eqnarray*}
0 &=& \sum_{i=1}^r \bar{\mu}_i \nabla g_i(\bar{x}) +  \sum_{j = 1}^s \bar{\kappa}_j \nabla h_j(\bar{x}), \\
0 &=& \bar{\mu}_i g_i(\bar{x}), \ \bar{\mu}_i \ge 0,  \ \textrm{ for } \ i = 1, \ldots, r.
\end{eqnarray*}
Since (MFCQ) holds at $\bar{x} \in \Omega,$ the gradient vectors $\nabla h_j(\bar{x}), j = 1, \ldots, s,$ are linearly independent and there exists a vector $v \in \Bbb{R}^n$ such that $\langle \nabla g_i(\bar{x}), v \rangle < 0, i \in \{i : g_i(\bar{x}) = 0\}$ and $\langle \nabla h_j(\bar{x}), v \rangle  = 0, j = 1, \ldots, s.$ Therefore
\begin{eqnarray*}
0 &=& \sum_{i=1}^r \bar{\mu}_i \langle \nabla g_i(\bar{x}), v \rangle  +  \sum_{j = 1}^s \bar{\kappa}_j \langle \nabla h_j(\bar{x}), v \rangle \\
&=& \sum_{i=1}^r \bar{\mu}_i  \langle \nabla g_i(\bar{x}), v \rangle.
\end{eqnarray*}
Then we deduce easily that $(\bar{\mu}, \bar{\kappa}) = (0, 0),$ which is a contradiction.
\end{proof}

Clearly, we can write
\begin{eqnarray*}
N(\Omega, x) &=& \left \{\sum_{i=1}^r \mu_i^ 2 \nabla g_i(x) +  \sum_{j=1}^s\kappa_j\nabla h_j(x) \ : \ \mu \in \mathbb{R}^r, \ \kappa\in \mathbb{R}^s, \ \mu_i g_i(x) = 0, \ i = 1, \ldots, r \right \}.
\end{eqnarray*}
Here, $\mu := (\mu_1, \ldots, \mu_r) \in \mathbb{R}^r$ and $\kappa := (\kappa_1, \ldots, \kappa_s) \in \mathbb{R}^s.$ For simplicity of notation, we put
$$a := (y^1, \ldots, y^{n + 1}, \mu^1, \ldots, \mu^{n + 1}, \kappa^1, \ldots, \kappa^{n + 1}, \lambda) \in \mathbb{R}^{n(n + 1)} \times \mathbb{R}^{r(n + 1)} \times \mathbb{R}^{s(n + 1)} \times \mathbb{R}^{n}.$$ 
For each $x \in \mathbb{R}^n,$ let 
\begin{eqnarray*}
A(x) := \Big\{ a 
&:& \nabla_y \phi(x, y^k) - \sum_{i = 1}^{r} [\mu^k_i]^2 \nabla g_i(y^k) -  \sum_{j = 1}^{s} \kappa^k_j \nabla h_j(y^k)  = 0, \ k = 1, \ldots, n + 1, \\
&& \mu^k_i g_i(y^k) = 0, \ k = 1, \ldots, n + 1, i = 1, \ldots, r,  \ y^1, \ldots, y^{n + 1} \in \Omega(x), \  \lambda \in \mathbf{P} \Big \},
\end{eqnarray*}
where we put
\begin{eqnarray*}
\mathbf{P}&:=& \{\lambda := (\lambda _1, \ldots, \lambda_{n}) \in \mathbb{R}^{n} : \lambda_k \ge 0 \textrm{ and } \sum_{k = 1}^n \lambda_k \le 1\}.
\end{eqnarray*}

The next two lemmas show the upper H\"older continuity of Lagrange multipliers.

\begin{lemma} \label{Lemma38}
Let \emph{(MFCQ)} hold on $\Omega.$ The following statements hold:
\begin{enumerate}
\item [{\rm (i)}] For each $x \in \mathbb{R}^n,$  $A(x)$ is a nonempty compact set.

\item[{\rm (ii)}] Let $\bar{x} \in \mathbb{R}^n.$  For each $\epsilon > 0$ there exist constants $c > 0$ and $\alpha > 0$ such that
\begin{eqnarray*}
A(x) \subset A(\bar{x}) + c\|x - \bar{x}\| ^\alpha \mathbb{B} \quad \textrm{ for all } \quad x \in \mathbb{B}_\epsilon(\bar{x}).
\end{eqnarray*}
\end{enumerate}
\end{lemma}
\begin{proof}
(i) The set $A(x)$ is obviously closed and it is bounded because of Lemma~\ref{Lemma37}.  Furthermore, by Lemma~\ref{Lemma36}, it is not hard to see that $A(x)$ is nonempty.

(ii) Take any $\epsilon > 0.$ Since (MFCQ) holds on $\Omega,$ we can find a constant $R > 0$ such that $A(x) \subset \mathbb{B}_R$ for all $x \in \mathbb{B}_\epsilon(\bar{x}).$ 

On the other hand, by definition, for each $x \in \mathbb{R}^n$  we have $y \in \Omega(x)$ if and only if $y \in \Omega$ and $\psi(x) = \phi(x, y),$ or equivalently
$$g_i(y) \leq 0, \ i=1, \ldots, r, \ h_j(y) = 0, \ j=1, \ldots, s, \textrm{ and } \psi(x) = \phi(x, y).$$
Therefore, we can write
\begin{eqnarray*}
A(x)  & = & \{a :  G_i(x, a) \le 0, i = 1, \ldots, \tilde{r}, \ \textrm{ and } \ H_j(x, a) = 0, j = 1, \ldots, \tilde{s} \}
\end{eqnarray*}
for some locally Lipschitz and semialgebraic functions $G_i$ and $H_j.$ Let
\begin{eqnarray*}
\Gamma(x, a) &:=& \sum_{i = 1}^{\tilde{r}} [G_i(x, a)]_+ + \sum_{j = 1}^{\tilde{s}} |H_j(x, a)|.
\end{eqnarray*}
Then $\Gamma$ is a locally Lipschitz and semialgebraic function and $A(x) = \{a : \Gamma(x, a) = 0\}.$ Since $\mathbb{B}_R$ is a compact set, it follows from the classical \L ojasiewicz inequality (see Theorem~\ref{ClassicalLojasiewiczInequality}) that there are constants $c > 0$ and $\alpha > 0$ such that
\begin{eqnarray*}
c\, \mathrm{dist}(a , A(\bar{x})) & \le & |\Gamma(\bar{x}, a)|^{\alpha} \quad \textrm{ for all } \quad a \in \mathbb{B}_R.
 \end{eqnarray*}

On the other hand, since $\Gamma$ is locally Lipschitz, it is globally Lipschitz on the compact set $\mathbb{B}_\epsilon(\bar{x}) \times \mathbb{B}_R;$  in particular, there exists a constant $L > 0$ such that
 \begin{eqnarray*}
|\Gamma(x, a) - \Gamma(\bar{x}, a)|  & \le & L \|x - \bar{x}\| \quad \textrm{ for all } \quad (x, a) \in \mathbb{B}_\epsilon(\bar{x}) \times \mathbb{B}_R.
\end{eqnarray*}

Let $x \in \mathbb{B}_\epsilon(\bar{x})$ and take an arbitrary  $a \in A(x).$ Then $A(x) \subset \mathbb{B}_R$ and $\Gamma(x, a) = 0.$ Therefore,
 \begin{eqnarray*}
c\, \mathrm{dist}(a , A(\bar{x}))
& \le & |\Gamma(\bar{x}, a)|^{\alpha} \ = \ |\Gamma(x, a) - \Gamma(\bar{x}, a)|^{\alpha} \\
& \le & L^{\alpha} \|x - \bar{x}\|^{\alpha}.
 \end{eqnarray*}
This implies immediately the required statement.
\end{proof}

For simplicity of notation, we write $b := (\mu^0_1, \ldots, \mu^0_r, \kappa^0_1, \ldots, \kappa^0_s) \in \mathbb{R}^r \times\mathbb{R}^s.$ For each $x \in \Omega$ and $R > 0,$ let
\begin{eqnarray*}
B_R(x) := \{b & : &  \textrm{there exists } v \in \partial^\circ \psi (x) \textrm{ such that } \\
&& v + \sum_{i = 1}^{r} [\mu^0_i]^2 \nabla g_i(x) + \sum_{j = 1}^{s} \kappa^0_j \nabla h_j(x) \in \mathbb{B}_R, \\
&& \mu^0_i g_i(x) = 0, \ i = 1, \ldots, r\}.
\end{eqnarray*}

\begin{lemma} \label{Lemma39}
Let \emph{(MFCQ)} hold on $\Omega$ and let $\bar{x} \in \Omega.$ For each $\epsilon > 0$ there exist constants $R > 0, c > 0$ and $\alpha > 0$ such that for all 
$x \in \Omega \cap \mathbb{B}_\epsilon(\bar{x}),$ the set $B_R(x)$ is nonempty compact and satisfies
\begin{eqnarray*}
B_R(x) \subset B_R(\bar{x}) + c\|x - \bar{x}\| ^\alpha \mathbb{B}.
\end{eqnarray*}
\end{lemma}
\begin{proof}
By Lemmas~\ref{Lemma31} and \ref{Lemma32}, there exists a constant $R > 0$ such that for all $x \in \mathbb{B}_{\epsilon}(\bar{x}),$ we have $\partial^0 \psi(x)$ is a nonempty compact subset of $\mathbb{B}_R;$ in particular, $0 \in B_R(x).$ Furthermore, in light of Lemma~\ref{Lemma37}, it is easy to see that the set $B_R(x)$ is compact.

On the other hand, it follows from Lemma~\ref{Lemma32} and the Carath\'eodory theorem that for each $x \in \mathbb{R}^n$ we have 
$v \in \partial^\circ \psi(x)$ if and only if there are $(\lambda_1, \ldots, \lambda_n)  \in \mathbf{P}$ and $y^1, \ldots, y^{n + 1} \in \Omega(x)$ such that
\begin{eqnarray*}
v &  = &  \sum_{k = 1}^{n} \lambda _k \nabla_x \phi(x, y^k) + (1 - \sum_{k = 1}^{n} \lambda _k) \nabla_x \phi(x, y^{n + 1}).
\end{eqnarray*}
Recall that $y \in \Omega(x)$ if and only if $y \in \Omega$ and $\psi(x) = \phi(x, y),$ or equivalently
$$g_i(y) \leq 0, \ i=1, \ldots, r, \ h_j(y) = 0, \ j=1, \ldots, s, \textrm{ and } \psi(x) = \phi(x, y).$$
Therefore, by definition, we can write
\begin{eqnarray*}
B_R(x)  & = & \{b \ : \ \exists z \in \mathbb{R}^m, G_i(x, z, b) \le 0, i = 1, \ldots, \tilde{r}, \ \textrm{ and } \ H_j(x, z, b) = 0, j = 1, \ldots, \tilde{s} \}
\end{eqnarray*}
for some locally Lipschitz and semialgebraic functions $G_i$ and $H_j.$ In particular, $B_R(x)$ is the image of the set
\begin{eqnarray*}
C(x)  & = & \{(z, b) \ : \ G_i(x, z, b) \le 0, i = 1, \ldots, \tilde{r}, \ \textrm{ and } \ H_j(x, z, b) = 0, j = 1, \ldots, \tilde{s} \}
\end{eqnarray*}
under the map $\mathbb{R}^m \times \mathbb{R}^{r + s} \to \mathbb{R}^{r + s}, (z, b) \mapsto b.$ As in the proof of Lemma~\ref{Lemma38}(ii), we can find constants $c > 0$ and $\alpha > 0$ such that 
\begin{eqnarray*}
C(x) \subset C(\bar{x}) + c\|x - \bar{x}\| ^\alpha \mathbb{B}  \quad \textrm{ for all } \quad x \in \mathbb{B}_\epsilon(\bar{x}).
\end{eqnarray*}
Take any $b \in B_R(x).$ Then there exists $z \in \mathbb{R}^m$ such that $(z, b) \in C(x).$ Therefore,
\begin{eqnarray*}
\mathrm{dist}(b, B_R(\bar{x})) &\le& \mathrm{dist}((z, b), C(\bar{x})) \ \le \ c\|x - \bar{x}\| ^\alpha,
\end{eqnarray*}
from which the desired conclusion follows.
\end{proof}

We can now finish the proof of Theorem~\ref{Theorem31}.

\begin{proof}[Proof of Theorem~\ref{Theorem31}]
Let $\bar{x} \in \Omega$ and fix a positive real number $\epsilon_1.$ By Lemmas~\ref{Lemma31}, \ref{Lemma32} and \ref{Lemma39}, there exists a constant $R > 0$ such that for all $x \in \mathbb{B}_{\epsilon_1}(\bar{x}),$ we have $\partial^0 \psi(x) \subset \mathbb{B}_R$ and $B_R(x)$ is nonempty compact. In particular, it holds that
\begin{eqnarray*}
\inf\{\|w\| \ : \ w \in \partial^0 \psi(x) + N(\Omega, x) \} &=& 
\inf\{\|w\| \ : \ w \in \left(\partial^0 \psi(x) + N(\Omega, x)\right) \cap \mathbb{B}_R \}.
\end{eqnarray*}
Therefore, in order to prove the inequality \eqref{PT2}, it suffices to consider vectors $w \in \partial^0 \psi(x) + N(\Omega, x)$ with $\|w\| \le R.$

Recall that we write
\begin{eqnarray*}
\lambda &:=& (\lambda_1, \ldots, \lambda_n) \in \mathbb{R}^n,\\
a &:=& (y^1, \ldots, y^{n + 1}, \mu^1, \ldots, \mu^{n + 1}, \kappa^1, \ldots, \kappa^{n + 1}, \lambda) \in \mathbb{R}^{n(n + 1)} \times \mathbb{R}^{r(n + 1)} \times \mathbb{R}^{s(n + 1)} \times \mathbb{R}^{n}, \\
b &:=& (\mu^0, \kappa^0) \in \mathbb{R}^{r} \times \mathbb{R}^{s}.
\end{eqnarray*}
Define the function $P$ by
\begin{eqnarray*}
P(x, a, b)
&:= & \sum_{k = 1}^{n} \lambda _k \left [ \phi(x, y^k) - \sum_{i = 1}^{r} [\mu^k_i]^2 g_i(y^k) - \sum_{j = 1}^{s} \kappa^k_j h_j(y^k) \right]  \\
&& + \ (1 - \sum_{k = 1}^{n} \lambda _k) \left [ \phi(x, y^{n + 1}) - \sum_{i = 1}^{r} [\mu^{n + 1}_i]^2 g_i(y^{n + 1}) - \sum_{j = 1}^{s} \kappa^{n + 1}_j h_j(y^{n + 1}) \right]  \\
&&  + \ \sum_{i = 1}^{r} [\mu^0_i]^2 g_i(x) + \sum_{j = 1}^{s} \kappa^0_j h_j(x).
\end{eqnarray*}
Then $P$ is a polynomial in $n(n + 3) + r(n + 2) + s(n + 2)$ variables of degree at most $d + 2.$ By Theorem~\ref{LojasiewiczGradientInequality}, for each $(\bar{a}, \bar{b}),$ there exist constants $c(\bar{a}, \bar{b}) > 0$ and $\epsilon(\bar{a}, \bar{b}) > 0$ such that for all $\|(x, a, b) - (\bar{x}, \bar{a}, \bar{b})\| \le \epsilon(\bar{a}, \bar{b}),$
\begin{eqnarray*}
\|\nabla P(x, a, b) \| & \ge & c(\bar{a}, \bar{b}) |P(x, a, b) - P(\bar{x}, \bar{a}, \bar{b})|^{1 - \alpha},
\end{eqnarray*}
where $\alpha := \frac{1}{\mathscr{R}(n(n + 3) + r(n + 2) + s(n + 2), d + 2)}.$

We have
\begin{eqnarray*}
A(\bar{x}) \times B_R(\bar{x}) & \subset & \bigcup \left \{(a, b) :  \|(a, b) - (\bar{a}, \bar{b})\| < \frac{\epsilon(\bar{a}, \bar{b})}{4} \right\},
\end{eqnarray*}
where the union is taken over all $(\bar{a}, \bar{b})$ in $A(\bar{x}) \times B_R(\bar{x}).$ By Lemmas~\ref{Lemma38}~and~\ref{Lemma39}, $A(\bar{x}) \times B_R(\bar{x})$ is nonempty compact, and so there exist $(\bar{a}^l, \bar{b}^l) \in A(\bar{x}) \times B_R(\bar{x})$ for $l = 1, \ldots, N,$ such that
\begin{eqnarray*}
A(\bar{x}) \times B_R(\bar{x}) & \subset & \bigcup_{l = 1}^N \left \{(a, b) :  \|(a, b) - (\bar{a}^l, \bar{b}^l)\| < \frac{\epsilon(\bar{a}^l, \bar{b}^l)}{4} \right\}.
\end{eqnarray*}

Let $\epsilon_2 := \min_{l = 1, \ldots, N} \epsilon(\bar{a}^l, \bar{b}^l) > 0$ and $c := \min_{l = 1, \ldots, N} c(\bar{a}^l, \bar{b}^l) > 0.$ By Lemmas~\ref{Lemma38} and \ref{Lemma39} again, there exists a positive constant $\epsilon \le \min\{\epsilon_1, \frac{\epsilon_2}{2}\}$ such that for all $x \in \Omega \cap \mathbb{B}_\epsilon(\bar{x}),$
\begin{eqnarray*}
\mathrm{dist}((a, b), A(\bar{x}) \times B_R(\bar{x})) &\le& \frac{\epsilon_2}{4} \quad \textrm{ for all } \quad (a, b) \in A(x) \times B_R(x).
\end{eqnarray*}

Take any $x \in \Omega \cap \mathbb{B}_\epsilon(\bar{x})$ and any $w \in \partial^\circ \psi(x) + N(\Omega, x)$ 
with $\|w\| \le R.$ It follows from Lemma~\ref{Lemma36} and the Carath\'eodory theorem that there exist $v \in \partial^\circ \psi(x)$ and $(a, b) \in A(x) \times B_R(x)$ satisfying the following conditions
\begin{eqnarray*}
w & = & v + \sum_{i = 1}^{r} [\mu^0_i]^2 \nabla g_i(x) + \sum_{j = 1}^{s} \kappa^0_j \nabla h_j(x), \\
v &  = &  \sum_{k = 1}^{n} \lambda _k \nabla_x \phi(x, y^k) + (1 - \sum_{k = 1}^{n} \lambda _k) \nabla_x \phi(x, y^{n + 1}) \\
0 & = & \nabla_y \phi(x, y^k) - \sum_{i = 1}^{r} [\mu^k_i]^2 \nabla g_i(y^k) -  \sum_{j = 1}^{s} \kappa^k_j \nabla h_j(y^k), \ k = 1, \ldots, n + 1, \\
0 & = & \mu^0_i g_i(x), \ i = 1, \ldots, r, \\
0 & = & \mu^k_i g_i(y^k), \ k = 1, \ldots, n + 1, \ i = 1, \ldots, r, \\
&& y^1, \ldots, y^{n + 1} \in \Omega(x), \quad \lambda \in \mathbf{P}.
\end{eqnarray*}
A direct computation shows that 
\begin{eqnarray*}
P(x, a, b) &=& \psi(x) \quad \textrm{ and } \quad \nabla P(x, a, b) \ = \ (w, 0, 0).
\end{eqnarray*}

On the other hand, since the set $A(\bar{x}) \times B_R(\bar{x})$ is nonempty compact, there is $(\bar{a}, \bar{b}) \in A(\bar{x}) \times B_R(\bar{x})$ such that
\begin{eqnarray*}
\|(a, b) - (\bar{a}, \bar{b})\| & = & \mathrm{dist}((a, b), A(\bar{x}) \times B_R(\bar{x})).
\end{eqnarray*}
There exists an index $l \in \{1, \ldots, N\}$ such that
\begin{eqnarray*}
\|(\bar{a}, \bar{b}) - (\bar{a}^l, \bar{b}^l)\| & < & \frac{\epsilon(\bar{a}^l, \bar{b}^l)}{4}.
\end{eqnarray*}
We have
\begin{eqnarray*}
\|({a}, {b}) - (\bar{a}^l, \bar{b}^l)\|
& \le & \|({a}, {b}) - (\bar{a}, \bar{b})\|  + \|(\bar{a}, \bar{b}) - (\bar{a}^l, \bar{b}^l)\| \\
& = & \mathrm{dist}((a, b), A(\bar{x}) \times B_R(\bar{x})) + \|(\bar{a}, \bar{b}) - (\bar{a}^l, \bar{b}^l)\| \\
& \le & \frac{\epsilon_2}{4} + \frac{\epsilon(\bar{a}^l, \bar{b}^l)}{4} \ \le \ \frac{\epsilon(\bar{a}^l, \bar{b}^l)}{2}.
\end{eqnarray*}
This implies that
\begin{eqnarray*}
\|(x, {a}, {b}) - (\bar{x}, \bar{a}^l, \bar{b}^l)\|
& \le & \|x - \bar{x}\| + \|(a, b) - \bar{a}^l, \bar{b}^l)\| \\
& \le & \epsilon + \frac{\epsilon(\bar{a}^l, \bar{b}^l)}{2} \ \le \
\frac{\epsilon_2}{2} +  \frac{\epsilon(\bar{a}^l, \bar{b}^l)}{2} \ \le \ \epsilon(\bar{a}^l, \bar{b}^l).
\end{eqnarray*}
Note that $P(\bar{x}, \bar{a}^l, \bar{b}^l)  =  \psi(\bar{x}).$ Therefore,
\begin{eqnarray*}
\|w\| \ = \ \|\nabla P(x, a, b) \|
& \ge & c(\bar{a}^l, \bar{b}^l) |P(x, a, b) - P(\bar{x}, \bar{a}^l, \bar{b}^l)|^{1 - \alpha} \\
&  = & c(\bar{a}^l, \bar{b}^l) |\psi(x) - \psi(\bar{x})|^{1 - \alpha} \\
& \ge & c |\psi(x) - \psi(\bar{x})|^{1 - \alpha},
\end{eqnarray*}
which completes the proof.
\end{proof}

\begin{remark}\label{NX31}{\rm
When the constraint set $\Omega$ is convex, a close look at the proof of Theorem~\ref{Theorem31} reveals that the exponent $\alpha$ in \eqref{PT2} can sharpen to 
$\alpha = \frac{1}{\mathscr{R}(2n + 2r + 2s, d + 2)}.$ This is due to the fact that in this case $\Omega(x)$ is a singleton for all $x \in \mathbb{R}^n,$ and so we can reduce variables in the set $A(x)$ and the polynomial $P.$ The details are left to the reader.
}\end{remark}

\section{Error bounds for the regularized gap function} \label{Section4}

In this section, we establish an error bound result for the regularized gap function $\psi$ associated with the variational inequality~\eqref{VI}.

Note by definition that $\psi$ is nonnegative on $\Omega.$ Assume that the set
\begin{eqnarray*}
\{x \in \Omega \ : \ \psi(x) = 0\}
\end{eqnarray*}
is nonempty.

\begin{theorem}\label{Theorem41}
Let {\rm (MFCQ)} hold on $\Omega.$ For any compact set $K \subset \Bbb{R}^n,$ there exists a constant $c > 0$ satisfying the following error bound
\begin{eqnarray*}\label{PT5}
c\, \mathrm{dist}(x, [x \in \Omega : \psi(x) = 0]) & \leq & [\psi (x)]^{\alpha} \quad \textrm{ for all }  \quad x \in \Omega \cap K,
\end{eqnarray*}
where $\alpha := \frac{1}{\mathscr{R}(n(n + 3) + r(n + 2) + s(n + 2), d + 2)}$ and the function $\mathscr{R}(\cdot, \cdot)$ is defined in \eqref{RFunction}.
\end{theorem}

\begin{proof}
Define for notational convenience $\mathscr{Z} := \{x \in \Omega \ : \ \psi(x) = 0\}.$ 
By using standard compactness arguments, it suffices to show, for each $\bar{x} \in \Omega,$ that there exist constants $c(\bar{x}) > 0$ and $\epsilon(\bar{x}) > 0$  such that
\begin{eqnarray*}
c(\bar{x})\, \mathrm{dist}(x, \mathscr{Z})  &\le& [\psi (x)]^{\alpha} \quad \textrm{ whenever } \quad
x \in \Omega \cap \mathbb{B}_{\epsilon(\bar{x})}(\bar x).
\end{eqnarray*}

Indeed, the statement is rather straightforward provided that $\bar{x} \not \in \mathscr{Z}$ (because of $\psi(\bar{x}) > 0$ and the function $\psi$ is continuous). Let us consider the case $\bar{x} \in \mathscr{Z},$ i.e., $\psi(\bar{x}) = 0.$ In view of Theorem~\ref{Theorem31}, there are constants $c > 0$ and $\epsilon > 0$ such that
\begin{eqnarray*} \label{Eqn32}
\inf\{\|w\| : w \in \partial^0 \psi(x) + N(\Omega, x) \} &\ge& c |\psi(x)|^{1 - \alpha} \quad \textrm{ for } \quad x \in \Omega \cap \mathbb{B}_{\epsilon}(\bar x).
\end{eqnarray*}
We will show that
\begin{eqnarray*}
\frac{c \alpha}{4}\, \mathrm{dist}(x, \mathscr{Z}) &\leq & [\psi (x)]^{\alpha} \quad \textrm{ for } \quad x \in \Omega \cap \mathbb{B}_{\frac{\epsilon}{2}}(\bar x).
\end{eqnarray*}

Arguing by contradiction, suppose that there exists $\bar{y} \in \Omega \cap \mathbb{B}_{\frac{\epsilon}{2}}(\bar x)$ such that
\begin{eqnarray*}
\frac{c \alpha}{4} \, \mathrm{dist}(\bar{y}, \mathscr{Z}) & > & [\psi (\bar{y})]^{\alpha}.
\end{eqnarray*}
Then $\bar{y} \not \in \mathscr{Z}$ and so $\psi (\bar{y}) > 0.$ Let us consider the continuous (semialgebraic) function
$$\theta \colon \Omega \rightarrow \mathbb{R}, \quad x \mapsto [\psi(x)]^{\alpha}.$$
Clearly, the function $\theta$ is locally Lipschitz on $\{x \in \Omega \ : \ \psi(x) > 0\}.$ Furthermore, we have
\begin{eqnarray*}
\inf_{x \in \Omega \cap \mathbb{B}_{{\epsilon}}(\bar x)} \theta(x) &=& 0 \ < \ \theta(\bar{y}) \ = \ [\psi (\bar{y})]^{\alpha} \ < \ \frac{c \alpha}{4} \, \mathrm{dist}(\bar{y}, \mathscr{Z}).
\end{eqnarray*}
Thanks to the Ekeland variational principle \cite{Ekeland1974}, there exists a point $\bar{z} \in \Omega \cap \mathbb{B}_{\epsilon}(\bar x)$ such that
\begin{eqnarray*}
\theta(\bar{z}) & \le & \theta(\bar{y}), \quad \|\bar{y} - \bar{z}\| \ < \ \frac{\mathrm{dist}(\bar{y}, \mathscr{Z})}{2},
\end{eqnarray*}
and $\bar{z}$ is a minimizer of the function
$$\Omega \cap \mathbb{B}_{{\epsilon}}(\bar x) \rightarrow \mathbb{R}, \quad x \mapsto \theta(x) + \frac{c \alpha}{2} \|x - \bar{z}\|.$$
By construction, then
\begin{eqnarray*}
\mathrm{dist}(\bar{z}, \mathscr{Z}) &\ge& \mathrm{dist}(\bar{y}, \mathscr{Z}) - \|\bar{y} - \bar{z}\| \ > \ \frac{\mathrm{dist}(\bar{y}, \mathscr{Z})}{2} \ > \ 0,
\end{eqnarray*}
which implies that $\bar{z} \not \in \mathscr{Z}$ and $\psi(\bar{z})> 0.$ Furthermore, we have
\begin{eqnarray*}
\|\bar{z} - \bar{x}\| & \le & \|\bar{y} - \bar{z}\| + \|\bar{y} - \bar{x}\| \\
& < & \frac{\mathrm{dist}(\bar{y}, \mathscr{Z})}{2} + \|\bar{y} - \bar{x}\| \\
& \le & \frac{\|\bar{y} - \bar{x}\|}{2} + \|\bar{y} - \bar{x}\| \\
& \le & \frac{\epsilon}{4} + \frac{\epsilon}{2}  < \epsilon,
\end{eqnarray*}
and so $\bar{z}$ is an interior point of the closed ball $\mathbb{B}_{\epsilon}(\bar{x}).$

We therefore deduce from Lagrange's multipliers theorem that
\begin{eqnarray*}
0 & \in & \partial \theta(\bar{z}) + \frac{c \alpha}{2}\mathbb{B} + N(\Omega, \bar{z}).
\end{eqnarray*}
Note that
\begin{eqnarray*}
\partial \theta(\bar{z}) &=& \alpha  [\psi(\bar{z})]^{\alpha - 1} \partial \psi(\bar{z}).
\end{eqnarray*}
Hence
\begin{eqnarray*}
0 &\in& \partial \psi(\bar{z}) + \frac{c}{2} [\psi(\bar{z})]^{1 - \alpha} \mathbb{B} + N(\Omega, \bar{z}) \\
&\subset & \partial^0 \psi(\bar{z}) + \frac{c}{2} [\psi(\bar{z})]^{1 - \alpha} \mathbb{B} + N(\Omega, \bar{z}).
\end{eqnarray*}
This implies that
\begin{eqnarray*}
\inf\{\|w\| \ : \ w \in \partial^0 \psi(\bar{z}) + N(\Omega, \bar{z}) \}  & \le & \frac{c}{2} \, [\psi(\bar{z})]^{1 - \alpha} \ < \
c \, [\psi(\bar{z})]^{1 - \alpha},
\end{eqnarray*}
which is a contradiction.
\end{proof}

The following example indicates that in general the error bound result in Theorem~\ref{Theorem41} cannot hold globally for all  $x \in \mathbb{R}^n.$
\begin{example}{\rm
Consider the variational inequality \eqref{VI} with data
\begin{eqnarray*}
F(x_1, x_2) &:=& (x_2 - 1, x_1 x_2 - 1) \quad \textrm{ and } \quad \Omega \ := \ \mathbb{R}^2.
\end{eqnarray*}
Take any $\rho > 0.$ Then it is easily seen that 
\begin{eqnarray*}
\psi(x) &=& \frac{1}{2\rho}\|F(x)\|^2 \ = \ \frac{1}{2\rho}\left[(x_2 - 1)^2 + (x_1 x_2 - 1)^2\right],
\end{eqnarray*}
and so $\psi^{-1}(0) = \{(1, 1)\}.$ Consider the sequence $x^k := (k, \frac{1}{k})$ for $k \ge 1.$ 
As $k \to +\infty,$ we have 
\begin{eqnarray*}
\psi(x^k) &=& \frac{1}{2\rho} \left(\frac{1}{k} - 1\right)^2 \ \to \frac{1}{2\rho}, \\
\mathrm{dist}(x^k, \psi^{-1}(0)) &=& \sqrt{(k - 1)^2 + \left(\frac{1}{k} - 1\right)^2} \ \to \ + \infty.
\end{eqnarray*}
It turns out that there cannot exist any positive scalars $c$ and $\alpha$ such that
\begin{eqnarray*}
c\, \mathrm{dist}(x^k, \psi^{-1}(0)) & \le & [\psi(x^k)]^{\alpha}  
\end{eqnarray*}
for all $k$ sufficiently large. Thus, a global error bound with the regularized gap function $\psi,$ even raised to any positive power, cannot hold in this case.
}\end{example}

We now assume that $\Omega$ is convex and let SOL be the solution set of the variational inequality~\eqref{VI}. In view of \cite[Theorem~10.2.3]{Facchinei2003}, 
$x \in \mathrm{SOL}$ if and only if $x \in \Omega \cap \psi^{-1}(0).$ Finally, recalling Remark~\ref{NX31}, the argument in the proof of Theorem~\ref{Theorem41} actually proves the following result. 

\begin{theorem} \label{Theorem42}
Let {\rm (MFCQ)} hold on $\Omega.$ If the constraint set $\Omega$ is convex, then for any compact set $K \subset \Bbb{R}^n$ there exists a constant $c > 0$ satisfying the following error bound
\begin{eqnarray*}
c\, \mathrm{dist}(x, \mathrm{SOL}) & \leq & [\psi (x)]^{\alpha} \quad \textrm{ for all }  \quad x \in \Omega \cap K,
\end{eqnarray*}
where $\alpha := \frac{1}{\mathscr{R}(2n + 2r + 2s, d + 2)}$ and the function $\mathscr{R}(\cdot, \cdot)$ is defined in \eqref{RFunction}.
\end{theorem}

\begin{remark}{\rm
As usual, the generality may exclude simple cases: the exponents in Theorems~\ref{Theorem41} and \ref{Theorem42} is not ``sharp'' because in the case $F$ is strongly monotone and $\Omega$ is closed convex, $\alpha = \frac{1}{\mathscr{R}(2n + 2r + 2s, d + 2)},$ while it is well-known that (see \cite{Huang2005, Ng2007, Taji1993, Yamashita1997-2, Wu1993}) the exponent equals $\frac{1}{2}$ in such a case. Thus, although our exponent estimate works for the general case, it may not be tight in particular settings. This calls for further improvements of the exponents obtained in the general polynomial variational inequalities.
}\end{remark}

\subsection*{Acknowledgments}
The authors wish to thank Guoyin Li for his helpful comments on this paper.


\begin{thebibliography}{10}

\bibitem{Auchmuty1989}
G.~Auchmuty.
\newblock Variational principles for variational inequalities.
\newblock {\em Numer. Funct. Anal. Optim.}, 10(9--10):863--874, 1989.

\bibitem{Bochnak1998}
J.~Bochnak, M.~Coste, and M.-F. Roy.
\newblock {\em Real Algebraic Geometry}.
\newblock Springer, Berlin, 1998.

\bibitem{Clarke1975}
F.~H. Clarke.
\newblock Generalized gradients and applications.
\newblock {\em Trans. Amer. Math. Soc.}, 205:247--262, 1975.

\bibitem{Clarke1983}
F.~H. Clarke.
\newblock {\em Optimization and Nonsmooth Analysis}.
\newblock John Wiley \& Sons, New York et al., 1983.

\bibitem{Acunto2005}
D.~D'Acunto and K.~Kurdyka.
\newblock Explicit bounds for the {{\L}}ojasiewicz exponent in the gradient
  inequality for polynomials.
\newblock {\em Ann. Pol. Math.}, 87:51--61, 2005.

\bibitem{Ekeland1974}
I.~Ekeland.
\newblock On the variational principle.
\newblock {\em J. Math. Anal. Appl.}, 47:324--353, 1974.

\bibitem{Facchinei2003}
F.~Facchinei and J.~S. Pang.
\newblock {\em Finite-Dimensional Variational Inequalities and Complementarity
  Problem, Volume I, II}.
\newblock Springer-Verlag, New-York, 2003.

\bibitem{Fukushima1992}
M.~Fukushima.
\newblock Equivalent differentiable optimization problems and descent methods
  for asymmetric variational inequality problems.
\newblock {\em Math. Program}, 53:99--110, 1992.

\bibitem{HP2017}
H.~V. H\`a and T.~S. Ph\d{a}m.
\newblock {\em Genericity in Polynomial Optimization}, volume~3.
\newblock World Scientific Publishing, 2017.

\bibitem{Harker1990}
P.~T. Harker and J.~S. Pang.
\newblock Finite-dimensional variational inequality and nonlinear
  complementarity problems: A survey of theory, algorithms, and applications.
\newblock {\em Math. Program.}, 48:161--220, 1990.

\bibitem{Huang2005}
L.~R. Huang and K.~F. Ng.
\newblock Equivalent optimization formulations and error bounds for variational
  inequality problems.
\newblock {\em J. Optim. Theory Appl.}, 125:299--314, 2005.

\bibitem{Kurdyka2014}
K.~Kurdyka and S.~Spodzieja.
\newblock Separation of real algebraic sets and the {{\L}}ojasiewicz exponent.
\newblock {\em Proc. Amer. Math. Soc.}, 142:3089--3102, 2014.

\bibitem{LiGuoyin2010-2}
G.~Li.
\newblock On the asymptotic well behaved functions and global error bound for
  convex polynomials.
\newblock {\em SIAM J. Optim.}, 20:1923--1943, 2010.

\bibitem{LiGuoyin2018}
G.~Li, B.~S. Mordukhovich, T.~T.~A. Nghia, and T.~S. Ph\d{a}m.
\newblock Error bounds for parametric polynomial systems with applications to
  higher-order stability analysis and convergence rates.
\newblock {\em Math. Program. Ser. B}, 168(1--2):313--346, 2018.

\bibitem{LiGuoyin2015}
G.~Li, B.~S. Mordukhovich, and T.~S. Ph\d{a}m.
\newblock New fractional error bounds for polynomial systems with applications
  to {{H}}\"{o}derian stability in optimization and spectral theory of tensors.
\newblock {\em Math. Program. Ser. A}, 153:333--362, 2015.

\bibitem{LiGuoyin2009}
G.~Li and K.~F. Ng.
\newblock Error bounds of generalized {{D}}-gap functions for nonsmooth and
  nonmonotone variational inequality problems.
\newblock {\em SIAM J. Optim.}, 20(2):667--690, 2009.

\bibitem{LiGuoyin2010-1}
G.~Li, C.~Tang, and Z.~Wei.
\newblock Error bound results for generalized {{D}}-gap functions of nonsmooth
  variational inequality problems.
\newblock {\em J. Comput. Appl. Math.}, 233(11):2795--2806., 2010.

\bibitem{Mordukhovich2006-1}
B.~Mordukhovich.
\newblock {\em Variational Analysis and Generalized Differentiation, I: Basic
  Theory; II: Applications}.
\newblock Springer, Berlin, 2006.

\bibitem{Ng2007}
K.~F. Ng and L.~L. Tan.
\newblock Error bounds of regularized gap functions for nonsmooth variational
  inequality problems.
\newblock {\em Math. Program.}, 110:405--429, 2007.

\bibitem{PHAMTS2018}
T.~S. Ph\d{a}m, X.~D.~H. Truong, and J.-C. Yao.
\newblock The global weak sharp minima with explicit exponents in polynomial
  vector optimization problems.
\newblock {\em Positivity}, 22(1):219--24, 2018.

\bibitem{Rockafellar1998}
R.~T. Rockafellar and R.~J.-B. Wets.
\newblock {\em Variational Analysis}.
\newblock Springer, Berlin, 1998.

\bibitem{Solodov2003}
M.~V. Solodov.
\newblock Merit functions and error bounds for generalized variational
  inequalities.
\newblock {\em J. Math. Anal. Appl.}, 287:405--414, 2003.

\bibitem{Taji1993}
K.~Taji, M.~Fukushima, and T.~Ibaraki.
\newblock A globally convergent {{N}}ewton method for solving monotone
  variational inequalities.
\newblock {\em Math. Program.}, 58:369--383, 1993.

\bibitem{Wu1993}
J.~H. Wu, M.~Florian, and P.~Marcotte.
\newblock A general descent framework for the monotone variational inequality
  problem.
\newblock {\em Math. Program.}, 61:281--300, 1993.

\bibitem{Yamashita1997-1}
N.~Yamashita and M.~Fukushima.
\newblock Equivalent unconstrained minimization and global error bounds for
  variational inequality problems.
\newblock {\em SIAM J. Control Optim.}, 35(1):273--284, 1997.

\bibitem{Yamashita1997-2}
N.~Yamashita, K.~Taji, and M.~Fukushima.
\newblock Unconstrained optimization reformulations of variational inequality
  problems.
\newblock {\em J. Optim. Theory Appl.}, 92:439--456, 1997.

\end{thebibliography}

\end{document}